\documentclass[11pt]{amsart}
\usepackage{amsfonts,amscd,amsthm,amsgen,amsmath,amssymb, hyperref}
\usepackage[all]{xy}
\usepackage[vcentermath]{youngtab}
\usepackage[letterpaper,hmargin=1.55in,vmargin=1.2in]{geometry}
\usepackage{graphicx, xcolor}
\usepackage{bm}	
\usepackage{amsmath}
\usepackage{amsthm}	
\usepackage{amsfonts}	
\usepackage{amssymb}
\usepackage{mathrsfs}
\usepackage{latexsym}
\usepackage{enumerate}
\usepackage{dsfont}
\usepackage{amssymb}

\theoremstyle{plain}
\newtheorem{theorem}{Theorem}[section]

\newtheorem{lemma}[theorem]{Lemma}
\newtheorem{proposition}[theorem]{Proposition}

\theoremstyle{definition}
\newtheorem{definition}[theorem]{Definition}

\theoremstyle{remark}
\newtheorem{example}{Example}[section]

\newtheorem{remark}[theorem]{Remark}

\numberwithin{equation}{section}
\numberwithin{figure}{section}

\newcommand\tr{\text{tr} }

\newcommand\RR{\mathbb R}
\newcommand\CC{\mathbb C}
\newcommand\ZZ{\mathbb Z}

\newcommand\NN{\mathbb N}

\begin{document}

\title[Projective Families of Dirac operators on a Banach Lie Groupoid]{Projective Families of Dirac operators on a Banach Lie Groupoid}

\author{Pedram Hekmati}
\address[Pedram Hekmati]{School of Mathematical Sciences, University of Adelaide,
Adelaide, SA 5005, Australia}
\email{pedram.hekmati@adelaide.edu.au}
\author{Jouko Mickelsson}
\address[Jouko Mickelsson]{Department of Mathematics and Statistics, University of Helsinki,
Finland}
\email{jouko.mickelsson@helsinki.fi}
\begin{abstract} We introduce a Banach Lie group $G$ of unitary operators subject to a natural trace condition. We compute the homotopy groups of $G$, describe its cohomology and construct an $S^1$-central extension.  We show that the central extension determines a non-trivial gerbe on the action Lie groupoid $G\ltimes  \frak  k$, where $\frak k$ denotes the Hilbert space of self-adjoint Hilbert--Schmidt operators. With an eye towards constructing elements in twisted K-theory, we prove the existence of a cubic Dirac operator $\mathbb D$ in a suitable completion of the quantum Weil algebra $\mathcal{U}(\frak{g}) \otimes Cl(\frak{k})$, which is subsequently extended to a projective family of self-adjoint operators $\mathbb D_A$ on $G\ltimes  \frak k$. While the kernel of $\mathbb D_A$ is infinite-dimensional, we show that there is still a notion of finite reducibility at every point, which suggests a generalized definition of twisted K-theory for action Lie groupoids.

\end{abstract}

\thanks{{\em Acknowledgements.}
This work is supported by the Australian Research Council Discovery Projects DE120102657, DP130102578, DP110100072 and the Academy of Finland Grant 1138810.}

\maketitle

\section{Introduction}
The present paper is motivated by an attempt to extend the representation theoretic construction
of twisted K-theory cocycles on  compact Lie groups to more general settings. Twisted K-theory has its origin in \cite{DK}, where it was defined in terms of bundles of Azumaya algebras. The twist in question is a torsion class in the degree three integral cohomology of the space. The restriction to torsion elements was lifted in \cite{R} by passing to infinite dimensional algebra bundles. Recent years have witnessed a resurgence of interest in twisted K-theory, none the least due to its various applications in conformal field theory. A particularly convenient model for twisted K-theory  is as homotopy classes of sections of an associated bundle of Fredholm operators $P\times_{PU(H)} \mathcal F$. Here $P$ is a principal $PU(H)$-bundle on the space determined by the twist and $PU(H)$ is the projective unitary group of a complex Hilbert space, acting on the space of Fredholm operators $ \mathcal F$ by conjugation. \\

In the case of a compact connected simple Lie group $G$ the basic ingredients in the construction of twisted K-cocycles are the positive energy
representations of a central extension $\widehat{LG}$ of the smooth loop group $LG$ \cite{Mi04}. The level $k$ of the representation is related by transgression to the Dixmier-Douady class in $H^3(G, \mathbb{Z})=\mathbb Z$, which
constitutes the twist in K-theory. The projective representations of $LG$ appear in the construction in two ways.
First, one fixes an arbitrary highest weight representation of $\widehat{LG}$ of level $k$ 
in a complex Hilbert space $V_{\lambda}.$ Then
in order to construct a cubic Dirac operator $\mathbb D$ on the loop group $LG$ one also needs a representation of the
Clifford algebra based on the vector space $L\mathfrak g$ with a fixed  nondegenerate invariant
bilinear form. The representations of the Clifford algebra actually arise from the quasi-free
representations of the canonical anticommutation algebra defined by the polarization of
$L\mathfrak{g}$ to negative and positive Fourier components.  The spin representation of $\mathfrak{so}(L\mathfrak{g})$  yields a projective representation $S_\rho$ of $LG$ of  level
$h^\vee$, the dual Coxeter number of $G$, and with weight $\rho$ equal to the half sum of the positive roots. The full Hilbert space $H$ is the tensor product
$H= V_{\lambda} \otimes S_\rho$ carrying a loop group representation of level $k +h^\vee.$\\

The remaining ingredient in the construction is the coupling of the Dirac operator $\mathbb D$ to
$\mathfrak{g}$-valued  connection 1-forms $A$ on a trivial $G$-bundle on the unit circle $S^1,$ imitating
the case of Dirac operators on finite dimensional manifolds coupled to connections on a
complex vector bundle. The resulting family $\mathbb{D}_A$ of Fredholm operators transforms
equivariantly under the  projective representation of $LG$ (of level $k+h^\vee$) and determines a cocycle in twisted K-theory
of the quotient stack $\mathcal A//LG$ of conjugacy classes in $G$, where $\mathcal A$ denotes the affine space of connections on which $LG$ acts by gauge transformations. 
In fact all generators in the equivariant twisted K-theory $K_G(G, k+h^\vee)$ can be constructed
in this way \cite{FHT}. By restricting to the based loop group $\Omega G$, one obtains elements in the ordinary twisted K-theory of the group $G=
\mathcal A/\Omega G$.\\

The representation theory of $\widehat{LG}$ is closely related to the representation theory
of the restricted unitary group $U_{res}$ \cite{PS}. Restricted means that the off-diagonal
blocks of a unitary
transformations $g$ in a polarized Hilbert space $H=H_+ \oplus H_-$ are Hilbert--Schmidt operators. The
Lie algebra $\mathfrak{u}_{res}$ has a central extension $\hat{\mathfrak u}_{res}$ given
by the Lundberg cocycle $\omega_L(X,Y) = \frac14 \text{tr}\, \epsilon[\epsilon, X][\epsilon, Y]$ where
$\epsilon$ with $\epsilon^2=1$ is the grading operator in $H_+\oplus H_-$ \cite{Lu}. In the purely
algebraic approach this is called the Kac--Peterson cocycle \cite{KP}. The basic representation
of $\widehat{LG}$ is then obtained by a restriction to the subgroup $LG\subset
U_{res}$ in a fixed unitary representation of $G$. More general highest weight representations are obtained by a reduction from
tensor powers of the basic representation of $\hat U_{res}$ \cite{PS}. \\

Thus a natural question arises, namely whether it is possible to construct twisted K-theory cocycles on a `universal' space
$\mathcal E/U_{res}$ for some appropriate contractible space $\mathcal E$ carrying a free smooth action of $U_{res}$, in analogy with the construction for $G=\mathcal A/\Omega G$. The answer is no and the obstruction comes from the 
spin representation of a certain Clifford algebra. What corresponds to the dual Coxeter
number $h^\vee$ in the case of a compact Lie group $G$ now diverges. This can be understood topologically as follows.
The space $\mathcal E/U_{res}$ is a classifying space for principal $U_{res}$-bundles and has the homotopy type
of $U(\infty),$ the inductive limit of the unitary groups $U(n).$ Thus we are constructing 
twisted K-theory on the group $U(\infty),$ with the dual Coxeter number $h^\vee = \lim_{n\to\infty} n =\infty.$ In the
finite dimensional case of $U(n)$ (or $SU(n)$) the twisted K-theory groups $K(U(n), k)$ vanish
for $0<k \leq h^\vee.$ Thus it is natural that one gets nothing in the case of $U(\infty).$ \\

Meanwhile, the above obstruction provides a hint of what could be done. The issue with the spin representation of $\mathfrak{ u}_{res}$ is that not every element is implementable as an automorphism of the spin module. The necessary and sufficient requirement is 
that the elements in $\mathfrak u_{res}$ satisfy the Hilbert--Schmidt condition defined by a polarization in the \it adjoint representation \rm of the Lie algebra \cite{Lu}. We find that the resulting condition is for $[D,X]$ to be Hilbert--Schmidt, where $D$ is an unbounded
self-adjoint operator with spectrum equal to $\mathbb Z$. For instance, it can be interpreted as the generator of rigid rotations 
on the unit circle. Since $D$ is unbounded, this imposes a non-trivial restriction on $\mathfrak u_{res}$. The Lundberg cocycle in the spin representation induces then the cocycle
$$\omega(X,Y) = \text{tr}_c \, X[D,Y]$$
on the smaller Lie algebra of implementable automorphisms; the conditional trace $\text{tr}_c$ means that it is computed
 in a basis where $D$ is diagonal. We denote our restricted Lie algebra simply by
$\mathfrak g$ and the corresponding infinite unitary group by $G.$ \\

Adhering to the above restriction, we can now repeat the cubic Dirac operator construction almost verbatim in the new setting, aside from two subtle points. Firstly, since $\frak g$ is a Banach Lie algebra it does not support a Clifford algebra. This problem is resolved by instead considering  the adjoint action of $\frak g$ on the Hilbert space of self-adjoint Hilbert--Schmidt operators $\frak k$ and interpreting the Dirac operator $\mathbb D$ as an element in the quantum Weil algebra $U(\frak g_\CC) \otimes Cl(\frak k_\CC)$. Secondly, for $\mathbb D$ to be well-defined it is necessary to perform a normal ordering regularization. For loop groups, this amounts to suppressing an infinite constant by shifting the bottom of the energy spectrum to zero. In our case the ramification is more drastic as it involves subtraction by an unbounded operator.\\

Treating elements in $\frak k$ as generalized gauge connections, we can couple the Dirac operator to $A\in \frak k$ and prove the equivariance of $\mathbb D_A$ with respect to the action of $G$ in a projective highest weight representation. However, this family is not quite Fredholm. The kernel of $\mathbb{D}_A$ can be infinite-dimensional, but the infinite dimensionality is controlled in the following way. At every point $A\in \mathfrak k$, the kernel carries a finite
number of irreducible representations of the quantum Weil algebra $U(\mathfrak g_A) \otimes Cl(\mathfrak g'_A), $
where $\mathfrak g_A \subset \mathfrak g$ is the Lie algebra of the isotropy group at that point
and $\mathfrak g'_A = \mathfrak g_A\cap \frak k$.  More generally, we have the following definition:

\begin{definition}\label{def} \it Let $\mathcal G=G\ltimes \frak k$ be an action Lie groupoid defined by a smooth action of a Fr{\'e}chet Lie group $G$ on a Fr{\'e}chet manifold $\mathfrak k.$ Let $c\in H^2(\mathcal G, S^1)$ be a cocycle on $\mathcal G$ defining a groupoid $S^1$-central extension $\hat{\mathcal G}$ and fix a unitary representation of $\hat{\mathcal G}$ on a complex separable Hilbert space $H$.  Then the generalized odd twisted  K-theory group $K^1(\mathcal G, c)$
is defined as homotopy classes of  maps $f\colon \mathfrak k\to End_*(H)$ into the space of self-adjoint operators on $H$, such that
\begin{enumerate}
\item[$(i)$]
$f(A^g) =  g^{-1} f(A) g$ with $A\in \frak k$ and $g\in\mathcal G$ acting projectively on $H$, 
\item[$(ii)$]
the kernel of $f(A)$ decomposes into a finite sum of
irreducible modules of  $U(\mathfrak h) \otimes Cl(\mathfrak h')$, where $\mathfrak h$ is
the Lie algebra of the stabilizer group at $A$ and $\mathfrak h'$ is the adjoint module of $\mathfrak h$
equipped with a real $\mathfrak h$-invariant inner product.
\end{enumerate}
\end{definition}
 
For unbounded operators, we consider the strong operator topology on the space of self-adjoint operators. The generalized even twisted K-theory group $K^0(\mathcal G, c)$ is defined in the usual way by introducing a $\mathbb Z_2$ grading operator $\Gamma$ and requiring that the operators
$f(A)$ anticommute with $\Gamma.$ In the case of a free $G$-action on $\frak k$, the above definition gives the standard
 definition of twisted K-theory on the manifold $\frak k/G.$ For $G= LH,$ the loop group of
 a compact Lie group $H$, and $\frak k$ the space of smooth $\mathfrak h$-valued 1-forms on the circle,
 one recovers the equivariant twisted K-theory of $H.$\\
  
An outline of the paper is as follows. In  Section 2 we introduce the infinite-dimensional
 unitary group $G$ and its central extension. The homotopy groups of $G$ are determined and 
 the low dimensional cohomology groups are discussed. In  Section 3 we recall some basic relations
 between characteristic forms on $\frak k/G$ and Lie algebroid cocycles on $G\ltimes \frak k$, when
 a Lie group $G$ acts freely on a manifold $\mathfrak k.$ The discussion serves as a motivation for the
 non-free case, when $\frak k/G$ is not a smooth manifold and one is confined to work with Lie groupoid cocycles instead of (de Rham) forms on the quotient.
 In Section 4 we refine the geometric construction of Lie group cocycles by Wagemann and Wockel
 \cite{W} to the case of action Lie groupoids. Besides serving as a method for integrating our Lie algebroid cocycle, this is important also in gauge theory where the breaking of chiral symmetries manifests itself as abelian extensions of the group of gauge transformations \cite{Mi85}.
 
 Finally, in Section 5 we present a construction of a cubic Dirac operator based on the
 representation theory of a central extension of $G.$ Although one can easily give a formal
 definition of the cubic Dirac operator imitating the finite-dimensional case, particular care is needed in infinite dimensions as the formal operator is an infinite sum in the quantum Weil algebra. The 
 divergencies can be avoided by working in a suitable completion and performing normal ordering twice; first in the definition of the spin 
 group generators which are quadratic expressions in the Clifford algebra, and second in the 
 construction of the cubic term which involves products of the spin operators with the generators of the Clifford algebra. We can then proceed to define a  family of cubic Dirac operators and we  show that our construction conforms to Definition \ref{def} for the case of a Lie groupoid central extension determined by the Lie algebra
 cocycle $\omega$.

\section{The Banach Lie group $G$ and its central extension}

Let $H$ be a complex separable infinite dimensional Hilbert space and $D$ a self-adjoint operator such that in an orthonormal basis $\{e_n\}_{n\in\mathbb Z}$ the action
is defined by $De_n = n e_n.$ Thus $D$ can be interpreted as the generator for rotations of the unit circle $S^1$ in the Hilbert space $L^2(S^1)$ of square
integrable complex functions.

Let $G$ be the group of unitary operators $g$ on $H$ such that $[D,g]$ is a Hilbert--Schmidt operator, 
$$G = \{g\in U(H) | \ ||[D,g]||^2_2=\tr |[D,g]|^2<\infty\}.$$
In particular, in the orthonormal basis $\{e_n\}$ the off-diagonal
part of $g$ is Hilbert--Schmidt.  The Lie algebra $\mathfrak g$ of $G$ consists of bounded skew-adjoint operators $X$ with the property that $[D,X]$ is Hilbert--Schmidt.
Then for any $X,Y \in \mathfrak g$  the conditional trace $\omega(X,Y)= \tr_c \, X [D,Y]$ is absolutely convergent; the conditional trace means that it is evaluated in the
basis $\{e_n\}$. We shall omit the subscript $c$ in the sequel. By a direct computation one observes that $\omega$ is a Lie algebra 2-cocycle, namely it is skew-symmetric and satisfies
$$ \omega(X, [Y,Z]) + \omega(Y,[Z,X]) + \omega(Z, [X,Y]) =0$$
for all $X,Y,Z \in \mathfrak g.$ In finite dimensions $\omega$ would be the coboundary of the linear form $- \tr\, D X$, but in the Hilbert space $H$ this trace does not converge. 

We define the topology on $G$ by the Hilbert--Schmidt norm on $[D,g]$ and the supremum norm on the diagonal matrix elements of $g$. 
Consequently the Hilbert--Schmidt norm of the off-diagonal matrix elements of $g$ is continuous in this topology. The Lie algebra $\frak g$ acquires a Banach structure from the norm 
$$||X|| = ||X_d||_\infty + ||[D,X]||_2$$
where $X_d$ is the diagonal part of $X$. Thus $G$ becomes a real Banach Lie group  in the standard way, by using the exponential map to construct a local chart near the identity element of $G$ and using left translations by elements in the group to obtain an atlas. Smoothness of the local group structure near the unit and of the transition functions is ensured by the Campbell--Baker--Hausdorff formula. We note further that $G$ embeds continuously as a subgroup of the restricted unitary group $U_{res}(H_+\oplus H_-)$, but it is not a normal subgroup. Here the polarization of $H$  is defined by the sign of $D.$ We shall not
make use of this embedding since the cocycle defining the standard central extension of $U_{res}$
when restricted to $G$ is not equivalent to $\omega$, see Remark \ref{lacoh} below. 

\begin{lemma}\label{cocycle} The 2-cocycle $\omega$ defines a non-trivial central extension of the Lie algebra $\mathfrak{g}.$ \end{lemma} 
\begin{proof} Restricted to finite rank operators in $\frak g$, $\omega$ is the coboundary of $$\theta(X) = -\tr\, DX + \lambda \tr\, X$$ for any complex number $\lambda$ and
this is the most general form of any  $\theta$ with $\delta\theta = \omega$. The 1-cochain $\theta$ cannot be extended to $\mathfrak g$ since for bounded diagonal matrices the
trace does not converge for any $\lambda \in \CC$. 
\end{proof} 
\begin{remark} \label{lacoh}
The second Lie algebra cohomology of $\frak g$ is infinite dimensional. In particular, $\omega_m(X,Y)=\tr \, X [D^{1/(2m+1)},Y]$ defines a non-trivial 2-cocycle on $\frak g$  for every $m\in\NN$ and since the operator  $D^{1/(2m+1)}-D^{1/(2n+1)}$ is unbounded for $n\neq m$, a similar argument as in the proof of Lemma \ref{cocycle} implies that these cocycles are distinct in $H^2(\frak g,\CC)$. Similarly one shows that  $\omega$ cannot be cohomologous to the restriction of the  Lundberg  cocycle $\omega_{\textit{\tiny{L}}}(X,Y)=\tr \, X \left[\epsilon,Y\right]$ to the Lie algebra $\frak g$. Here $\epsilon= \frac{D}{|D|}$ is the sign operator and we recall that $\tr$ denotes the conditional trace. The zero mode of $D$ could belong to either the positive or the negative sector, as these yield cohomologous cocycles.
\end{remark}
A central extension $\widehat G$ of $G$ can be constructed  analogously to the method used in \cite{PS} for the central extension of the restricted unitary group.
Consider the set $S$ of pairs $(g,q)$ where $g\in G$ and $q$ is any invertible diagonal matrix such that the diagonal part of $gq^{-1}$ differs from the unit
matrix by a matrix $h$ such that $Dh$ is trace-class. Define an equivalence relation $(g,q) \sim (g',q')$ if $g=g'$ and $\text{det}((q^{-1} q')^D) =1$  and
set $\widehat G = S/{\sim}$. Here $q^D$ for a diagonal matrix $q$ means the operator with the matrix entries $(q^D)_{ii} = (q_{ii})^i$. The group law in $\widehat G$ is defined by $(g_1, q_1) (g_2, q_2) = (g_1g_2, q_1q_2).$  This is well-defined since the function $g\mapsto \text{det}(g^D)$
is well-defined and multiplicative for invertible diagonal matrices $g$ with $D(g-1)$ trace-class.

Near the unit element in $G$ we have a smooth local section $\psi\colon g\mapsto (g, g_d)$ where $g_d$ is the diagonal part of $g.$ This identifies locally the extension $\widehat G$ as a product
$G\times \mathbb C^*$ by $(g,q)\mapsto (g, \text{det}((qg_d)^D)).$ Using the first terms in the
Baker--Campbell--Hausdorff formula, $e^X e^Y = e^{X+Y +\frac12 [X,Y] +\dots} $ and  writing 
$$(e^X,\lambda) (e^Y,\mu) = (e^X e^Y,\lambda\mu e^{\frac12 \omega(X,Y)+\dots}),$$
where the dots signify higher order terms in $X,Y$, one checks that the Lie algebra cocycle arising from the central extension $\widehat G$ is indeed the 2-cocycle $\omega$.
In the above construction we could have taken the determinant $\text{det}((q^{-1} q')^D)$ to any integer power $k$ and then the corresponding Lie algebra cocycle
would be of the form
$$\omega(X,Y)= k\, \text{tr} X[D,Y],$$
a central extension of {\it level} $k.$

Recall that any Lie algebra 2-cocycle defines a closed left-invariant 2-form on a corresponding Lie group $G$ by identifying the Lie algebra with left-invariant vector fields on $G.$ 

\begin{lemma}\label{trivial}  The left-invariant 2-form $\omega$ is trivial in the real cohomology of $G$.
\end{lemma} 
\begin{proof}  It suffices to show that the pairing of $\omega$ with any smooth cycle in $G$ vanishes. Let $\Sigma \subset G$ be a smooth closed surface in $G.$  Then for any $g\in \Sigma$ the off-diagonal part of the infinite unitary matrix
$$f_N(g)_{ij}= \begin{cases} g_{ij} & \text{if }  |i|> N \text{ or } |j| > N \\ \delta_{ij} & \text{otherwise}\end{cases}$$ 
is smaller than $1/2$ in the Hilbert--Schmidt norm when $N=N(g)$ is large enough. By the continuity of the Hilbert--Schmidt
norm of the off-diagonal matrix elements of $g$, there is an open neighbourhood
$U_N(g)$ of $g$ such that the  Hilbert--Schmidt norm of $f_N(g)$ away from the diagonal is less than $1/2$ in $U_N(g).$ The family $\{U_N(g)\}_{g\in\Sigma}$ forms an open cover of the compact 
set $\Sigma,$ thus by choosing a finite subcover indexed by $N(g_1), \dots, N(g_p)$, we get a largest integer $N =\max\{N(g_i)\}$ such that the
Hilbert--Schmidt norm of the off-diagonal part of $f_N(g)$ is smaller than $1/2$ for all $g\in \Sigma.$ By the unitarity of $g$ this implies that $|g_{ii}-1| < 1/4$
for all $g\in \Sigma$ and $|i| > N.$ In particular, each $g_{ii}$ is invertible for $|i| > N.$ 

Define a 1-form $\theta$ on $\Sigma$ by
$$\theta(X) = \sum_{|k| > N}  k( X_{kk}  - g_{kk}^{-1} dg_{kk}(X) )  + \sum_{|k| \leq N} kX_{kk}$$
where the matrix $X\in \mathfrak g$  is again interpreted as a left invariant vector field on $G.$  Since $dg_{kk}(X) = g_{kk} X_{kk}
+ \sum_{j\neq k} g_{kj} X_{jk}$ we have $kg_{kk}^{-1} dg_{kk}(X) = kX_{kk} + \sum_{j\neq k} k g_{kk}^{-1} g_{kj} X_{jk}. $ Inserting the right hand side of this expression into the sum  $|k| > N$ above, we see that the second term is convergent by the Hilbert--Schmidt property of $[D,X]$  in the off-diagonal part of $g.$  The first term cancels against the diverging part of $\tr\, DX$ and so $\theta$ is well-defined. 

Since each $g_{kk}^{-1} dg_{kk}$ is closed,
this `renormalization' of the diverging 1-form $\tr\, DX$ does not affect the relation $\omega= d\theta$ on $\Sigma$  and therefore the integral of $\omega$ over the closed surface
$\Sigma$ is indeed zero. 
\end{proof} 

\begin{remark} \label{extension} Recall that if a group $G$ is not simply connected it could admit several inequivalent central extensions with the same Lie algebra extension.  For instance, the
simply connected covering group $\tilde G$ is an extension of $G$ but of course their Lie algebras are isomorphic when
the fundamental group $\pi_1(G)$ is discrete.  In Proposition \ref{homotopy} we show that in our case the group $G$  has a free abelian fundamental group on a set with the cardinality of $\RR$. Thus, the central extension $ \widehat G $ constructed above may be topologically non-trivial despite the vanishing of $\omega$ in $H^2(G,\RR)$. In addition, since $H_1(G,\mathbb Z) = \pi_1(G)$ it follows by the universal coefficient theorem that $H^2(G,\mathbb Z)$ is torsion-free.
\end{remark} 
 
\begin{proposition}\label{cohomology}   The cohomology groups $H^n(G,\RR)$ for $n\geq 1$ have rank greater than or equal to $2^{\aleph_0}$. 
\end{proposition} 
\begin{proof}  
By H\"older's inequality all commutators in the Lie algebra $\mathfrak g$ have the property of being trace-class and in particular the elements in the diagonal of $[X,Y]$ form a sequence in the Banach space $\ell_1$ of absolutely summable sequences. On the other hand, there is an infinite number of continuous linear functionals $\lambda$ on the Banach space $\ell_{\infty}$ of bounded sequences that vanish on
the vector subspace $\ell_1$.\footnote{Recall that the second dual of $\ell_1$ is the Banach algebra of regular Borel measures on the Stone--\v{C}ech compactification $\beta \NN$ of  the natural numbers. Elements in $\ell_\infty^*$ that vanish on $\ell_1\subset \ell_{\infty}$ can be constructed as follows. Consider the closed subspace  $c\subset \ell_\infty$ of bounded convergent sequences and its closed subspace $c_0\subset c$ of sequences converging to zero. By the Hahn--Banach theorem, the limit functional $(a_k) \mapsto \lim_{k\to\infty}a_k$ on $c$ can be extended to a generalized limit functional $\lambda\in \ell_\infty^*$. This is a non-trivial continuous linear functional of operator norm 1 that annihilates all of $c_0$  and hence vanishes on $\ell_1 \subset c_0$.} Any such linear functional defines a trace on the Lie algebra $\mathfrak g$ by applying $\lambda$ to the bounded sequence of diagonal
matrix elements of $X\in \mathfrak g.$  

Using this trace one can define a closed 1-form on $G$ by
$$\theta_{\lambda}= \lambda(g^{-1} dg),$$ 
where $\lambda$ is computed on the diagonal entries in the Lie algebra. It is not difficult to show that
$\theta_\lambda$ is not an exact form for $\lambda \neq 0$ by pairing with a cycle in $H_1(G,\ZZ)\cong \ZZ^{2^{\aleph_0}}$. Moreover, by $\theta_\lambda-\theta_{\lambda'}=\theta_{\lambda-\lambda'}$  it follows that these 1-forms are cohomologous if and only if $\lambda = \lambda'$. Since any trace functional $\lambda \in \ell^*_\infty$ can be scaled by a real number, we conclude that $\dim H^1(G,\RR) \geq 2^{\aleph_0}$. 	

Next any wedge product $\theta_{\lambda} \wedge \theta_{\lambda'}$ defines a closed 2-form on $G$, which is not exact since pairing with a 2-torus in $G$ yields a non-zero number in general, namely $(2\pi)^2$ times the form $\lambda\wedge \lambda' \in \bigwedge^2 \ell_\infty^* $ evaluated on some  $X,Y\in\frak g$ with integral diagonal entries. More specifically, a pair of linearly independent trace functionals can be constructed by defining $\lambda(X)= \lambda'(Y)=1$ and $\lambda(Y)= \lambda'(X)=0$ on linearly independent elements $X,Y\in\frak g$ whose diagonal entries form integral sequences in $\ell_\infty$ modulo $c_0$. The integrality condition ensures that the diagonals exponentiate to a torus in $G.$ We require
that $\lambda,\lambda'$ vanish on $c_0$ in order that the 2-form $\theta_{\lambda} \wedge \theta_{\lambda'}$ is closed. Now $\lambda, \lambda'$ defined on $\text{span}_\CC\{X,Y\} \oplus c_0$
extend to bounded linear functionals on the whole space $\ell_{\infty}$  by the Hahn--Banach theorem.

As before, the difference $\theta_{\lambda} \wedge \theta_{\lambda'} -\theta_{\lambda} \wedge \theta_{\lambda''} = \theta_{\lambda} \wedge \theta_{\lambda'-\lambda''}$  cannot be exact unless  $\lambda'-\lambda''$ is proportional to $\lambda$. Thus by considering real multiples of $\theta_{\lambda} \wedge \theta_{\lambda'}$,  it follows  that $\dim H^2(G,\RR) \geq 2^{\aleph_0}$. Since there exist infinitely many linearly independent trace functionals in $\ell^*_\infty$, this process can be continued inductively and we conclude that the dimension of $H^n(G,\RR)$ is larger than or equal to the cardinality of the continuum for all $n\geq 1$. 
\end{proof}  
\begin{remark}  By $H^1(G,\ZZ) = \text{Hom}_\ZZ(H_1(G,\ZZ),\ZZ) = \text{Map}(\RR,\ZZ)$ it follows that the rank of the first cohomology group of $G$ is $2^{2^{\aleph_0}} $, so in this case the lower bound in Proposition \ref{cohomology} is strict.  
\end{remark}
 
Let $L_{p}$ denote the Schatten two-sided  $*$-ideal of bounded operators $X\in  B(H)$ satisfying $\text{tr}  |X|^{p} < \infty$.
Recall that the stable unitary group $U(\infty)$ is a dense subgroup of the Banach Lie groups $U_p(H)= U(H)\cap (1+L_{p})$ for $p\geq 1$ and  these groups all have the same homotopy type. In particular, the homotopy groups are trivial in even degrees and freely generated by a single element in odd degrees. For instance, an explicit generator for the fundamental group of $U_p(H)$ is given by choosing a unit vector $v\in H$ and setting 
$\gamma\colon S^1\to U_p(H), \ z \mapsto 1 + (z-1)P_v$, where $P_v$ is the projection onto the span of $v$.

\begin{proposition} \label{homotopy} The homotopy groups $\pi_k(G)$ are equal to the homotopy groups of the stable unitary group $U(\infty)$ for $k\neq 1$, namely
$\pi_{2m}(G) =0$ and $\pi_{2m+3}(G) = \Bbb Z$ for $m\in \NN$. The fundamental group is the infinite free abelian group $$\pi_1(G)= \pi_1(G_d/G_d\cap G_0)\oplus\ZZ,$$
where $G_d \subset G$ is the subgroup of all unitary diagonal matrices and $G_0\subset G$ is the subgroup 
of elements $g$ for which the diagonal sum $\sum_{i\in\ZZ} |g_{ii} -1|$ converges.
\end{proposition} 
\begin{proof}  By Theorem $B$ in \cite{P} the group $G_0$ has the same homotopy type as the group $U(\infty)$ of infinite unitary matrices which differ from the unit
matrix by a matrix of finite size. 
Now $G_0$ is a normal subgroup in $G$ and in fact $G/G_0$ is abelian. The subgroup $G_d$ of diagonal matrices defines a Banach Lie group $G_d/(G_0\cap G_d)$ which is homotopy equivalent 
 to $G/G_0.$  This follows
using the argument in the proof of Lemma \ref{trivial}. Namely, any $g\in G$ has the property that the diagonal matrix elements are nonzero outside a finite block with matrix
indices $|i|, |j| < N$ and furthermore $|g_{ii}| \to 1$ as $|i| \to \infty.$ Multiplying $g$ by the diagonal matrix $d$ with $d_{ii} = g_{ii}^{-1}$ for $|i| >N$ and $g_{ii} =1$ otherwise, one obtains $dg\in G_0.$ Although $d$ is not unitary in general we can use the homotopy equivalence of the unit circle with $\mathbb C^{\times}$ to conclude that
the class $gG_0$ is represented as $d'G_0$ with $d'_{ii} = d_{ii}/|d_{ii}|$ and the inclusion $G_d/(G_0\cap G_d) \to G/G_0$ is an isomorphism. 

Using the fact that any continuous loop in $S^1$ is homotopic to a homomorphism $S^1 \to S^1$ we conclude that any based loop in $G_d$ is homotopic 
to a map $t\mapsto (e^{it a_k})$ with $a_k\in \mathbb Z$ and $0\leq t\leq 2\pi.$ The infinite sequence $(a_k)_{k\in\mathbb Z}$ must be bounded in order that the
loop is continuous in the operator norm topology. The set of the bounded sequences has the cardinality of the set of real numbers and thus
$\pi_1(G_d)$  is a free abelian group of cardinality $2^{\aleph_0}$.
The continuous loops in the subgroup $G_d\cap G_0$ correspond to sequences $(a_k)$ such that $a_k=0$ when $|k| >> 0.$ This set has the cardinality of
the set $\mathbb Q$ of rational numbers and thus  $\pi_1(G_d/ G_d\cap G_0)$ is the additive abelian group of bounded sequences of integers modulo
sequences with finite nonzero entries and has the cardinality of  $\mathbb R.$ 

Next, using the homotopy equivalence of $G/G_0$ and $G_d/(G_d\cap G_0)$ and the long exact sequence
$$ \cdots \to \pi_k(G_d\cap G_0) \to \pi_k(G_d) \to \pi_k(G/G_0) \to \pi_{k-1}(G_d\cap G_0) \to \dots$$
we conclude that $\pi_k(G/G_0) =0$ for $k> 1$.  Consequently,  the homotopy exact
sequence 
$$ \cdots \to \pi_k(G_0) \to \pi_k(G) \to \pi_k(G/G_0) \to \pi_{k-1}(G_0) \to \cdots$$
yields  $\pi_k(G) = \pi_k(G_0)$ for $k>1$ and a short exact sequence 
$$0\to \ZZ\to \pi_1(G) \to \pi_1(G_d/ G_d\cap G_0) \to 0.$$
This group extension is split since $G$ is an $H$-space, so its fundamental group is abelian and consequently the extension must be central. Furthermore, $\pi_1(G_d/ G_d\cap G_0)$ is a free abelian group so by lifting each generator, the isomorphism $\pi_1(G_d/ G_d\cap G_0) \to \pi_1(G)/\ZZ$ extends to a homomorphism $\pi_1(G_d/ G_d\cap G_0) \to \pi_1(G)$ and the obstructing 2-cocycle must vanish. Thus we have $\pi_1(G) = \pi_1(G_d/ G_d\cap G_0)\oplus \ZZ$.
 
\end{proof}  
\begin{remark}  The homotopy groups of $G$ agree with the homotopy groups of the product
$G_0 \times K(L,1)$, where $L=\pi_1(G_d/ G_d\cap G_0)$ and  $K(\pi,n)$ denotes the Eilenberg-MacLane space whose only non-trivial homotopy group is equal to the group $\pi$ in dimension $n$.
However, in order to use the Whitehead theorem to conclude that $G$ is homotopy equivalent to
$G_0 \times K(L,1)$, we would need a continuous map between the two spaces that induces the isomorphism between the homotopy
groups.  We observe further that the universal covering group of $G$ has the same homotopy groups as $\widetilde U_p(H) =\{(g,z)\in U_p(H)\times \CC  | \  \det_p(g)=e^z\}$ for any $p\geq 1$, where $\det_p$ is the Carleman--Fredholm determinant of order $p$,
$$\text{det}_p(g) = \det\left(ge^{\sum_{j=1}^{p-1}(-1)^{j}\frac{(g-1)^j}{j}}\right).$$
\end{remark}

\section{K-theory and Lie algebroid cocycles}
In this section we discuss the relation between Lie algebroid cocycles and the Chern character of classes in odd K-theory. We then show that our Lie algebra cocycle $\omega$ defines a non-trivial cocycle on a natural action Lie groupoid associated to the group $G$. 
As a motivation consider first the following situation. Let $\frak k$ be a contractible manifold and $G$ a Lie group acting
freely on $\frak k$.  The quotient $\frak k/G$ is then a smooth manifold. If $\Omega$ is a closed integral 3-form on the base $\frak k/G$, 
the pullback $\pi^*\Omega$ with respect to the canonical projection $\pi\colon \frak k \to \frak k/G$ is an exact form on $\frak k,$  $\pi^*\Omega = d\theta.$ The form $\theta$ is
closed along the $G$-orbits in $\mathfrak k,$ therefore it defines a (possibly exact) 2-cocycle on the Lie algebra $\mathfrak g$ with coefficients in the algebra of smooth functions on $\frak k$, or equivalently, a cocycle in the degree two cohomology  of the action Lie algebroid $\frak g\ltimes \frak k$ with values in the sheaf of smooth $\RR$-valued functions.  In particular, for a connected simply connected (infinite-dimensional) Lie group  one has $H^2(G,\mathbb Z)=\pi_2(G)$ and by the exact homotopy sequence for fibrations  $\pi_1( \frak k/G)= \pi_2( \frak k/G)=0$, which implies $H^3( \frak k/G,\mathbb Z)= \pi_3( \frak k/G)= \pi_2(G) = H^2(G, \mathbb Z)$ and the transgression discussed above is a 
realization of this isomorphism. This happens for instance when $G$ is the based loop group of a connected simply connected Lie group $H$ and $\frak k$ is the  space of smooth $\mathfrak h$-valued 1-forms on the unit circle on which $G$ acts by gauge transformations.

In the opposite direction, any such Lie algebroid 2-cocycle $\theta$ defines a class in $H^3( \frak k/G, \mathbb Z).$ Although the 3-form is not canonically defined, the class
is determined as follows. Let $s_3$ be a closed singular 3-chain on $\frak k/G$ and choose a lift $\hat s_3$ to $\frak k.$ Then $\hat s_3$ is not closed in general, but
the boundary $\partial\hat s_3 = b_2$ projects to zero as a singular 2-simplex in $\frak k/G,$ namely $b_2$ is a vertical 2-cycle and can be paired with the vertical 2-cocycle 
$\theta$. We can define a 3-form $\Omega$ on $\frak k/G$ by 
$$ \langle \Omega, s_3\rangle = \langle\theta, b_2\rangle.$$
In our case, the $G$-action is not free and the quotient $\frak k /G$ is a differentiable stack, but we can still think of the gerbe
as an $S^1$-central extension of the action Lie groupoid $G\ltimes \mathfrak k$ defined by integrating the cocycle $\theta\in H^2(\frak g\ltimes \frak k,\RR)$ to an $S^1$-valued groupoid cocycle \cite{XL-G}. Note that while we have defined $\theta$ as a vertical de Rham cocycle, one could more generally  consider $\theta$ as a vertical singular cocycle.

Degree two cocycles for infinite-dimensional algebras, including current algebras and algebras of vector fields, typically arise from quantization of families of Dirac type operators.
The 2-cocycles are then determined by the K-theory class of a family of Fredholm operators and can be computed from index theory \cite{CM}.
We want to point out that higher order cocycles are generated in the same way.

Let $\mathcal F_{*}$ be the space of  self-adjoint
Fredholm operators acting in a complex Hilbert space $H,$  with both positive and essential spectrum. The bounded Fredholm operators provide a model for the classifying space of odd complex K-theory but in many applications one has to
deal with unbounded operators. There are then different options for the topology on $\mathcal F_{*}.$ The simplest is the Riesz
topology defined 
by the map $D\mapsto D/( |D|^2 +1)^{1/2}$ to bounded Fredholm operators and defining the topology by pullback from the operator norm topology. Another popular choice is the \it gap topology, \rm see \cite{BLP} for details in the context of K-theory.  
Alternatively, one can use the group $U_1(H)$ of unitary operators $g$ in $H$ such that $g-1$ is a trace-class operator  as a classifying space \cite{AS}. 

Typically a representative for a class in $K^1(\mathfrak k/G)$ is given by a continuous $G$-equivariant map $f\colon \mathfrak k \to U_1(H)$, that is, $f(A^g) = g^{-1} f(A) g$ for $g\in G$ with a fixed unitary representation of  $G$ on $H.$  However, we have
\begin{lemma} Any $G$-equivariant map $f\colon\mathfrak k \to U_1(H)$ is homotopic to a $G$-invariant map. \end{lemma} 

\begin{proof} The  $G$-action on $\mathfrak k \times U(H)$ given by $(A, u) \mapsto (A^g, g^{-1} u g)$ defines a principal $U(H)$ bundle over $ \mathfrak k/G.$ This
bundle is trivial by Kuiper's theorem, so there exists a trivialization, given by a map $r\colon \mathfrak k \to U(H)$ such that $r(A^g) = g^{-1} r(A).$ Let $r_t$ with
$0\leq t\leq 1$ be a contraction of this map, $r_0(A) =1$ and $r_1(A) =r(A).$ Define $f_t(A) = r_t(A)^{-1} f(A) r_t(A).$ Then $f_0 = f$ and $f_1(A^g) = f_1(A)$
is an invariant map. 
\end{proof}

Using $U_1(H)$ as the classifying space for the odd $K$-group and a $G$-invariant representative of an element $f$ in $K^1(\mathfrak k/G)$, its Chern character is given by pulling back the odd generators $\rho_{2k+1}= \text{tr}\, (g^{-1}dg)^{2k+1}$ of the cohomology of $U_1(H)$ by the map $f$. If $f$ is not $G$-invariant, the forms $f^*\rho_{2k+1}$ are no longer basic forms on the bundle $\mathfrak k \to \mathfrak k/G.$ However, the pullbacks are still $G$-invariant and furthermore they vanish when all the arguments are
vertical vector fields; this follows from the formula
$$ \tr\, [(g^{-1}fg)^{-1} d(g^{-1}fg)]^{2k+1} = \tr\, (f^{-1}df)^{2k+1} + \dots$$
where the dots represent a differential polynomial in the Maurer-Cartan 1-forms $f^{-1} df$ and $g^{-1} dg$ which is at least of degree one in $g^{-1}dg$. Let $f^{*} \rho_{2k+1}=db_{2k}$ and let $\theta_{2k}$ be the restriction of the form $b_{2k}$ in the vertical directions on $\mathfrak k.$ Then $\theta_{2k}$ defines an $\RR$-valued cocycle of degree $2k$ on the Lie algebroid $\frak g\ltimes \frak k$. 
This proves

\begin{proposition} Let $G$ be a Lie group that acts smoothly on a contractible manifold  $\frak k$. Let $f\colon \frak k \to \mathcal F_{*}$ be a $G$-equivariant family representing an element in $K^1(\mathfrak k/G)$. Then the Chern character of $f$ is
given by the even cocycles $\theta_{2k}$ in the Lie algebroid cohomology $H^{2k}(\frak g\ltimes \frak k,\RR)$. 
\end{proposition}

Returning to our  setting, let $\frak k$ denote the real Hilbert space of self-adjoint Hilbert--Schmidt operators with the bilinear form $\langle X,Y\rangle=\tr XY$. Then $\frak k$ is a contractible space carrying a non-free smooth affine action by our Banach Lie group $G$, 
$$G\times \frak k \to \frak k, \ \ \ (X,g) \mapsto g^{-1} X g + g^{-1} [D,g].$$ This action is well-defined as $\frak k$ is a two-sided ideal inside the algebra of bounded operator $B(H)$. Let $G\ltimes \frak k $ denote the associated action Lie groupoid and $\frak g\ltimes \frak k$ its  Lie algebroid.
We have seen that the Lie algebra cocycle $\omega$ can be integrated to a locally smooth group 2-cocycle defining a central extension $\widehat G$ of $G.$ We can view this also as an $S^1$-extension of the groupoid
$G \ltimes \frak k.$ While the central extension $\widehat G$ is non-trivial in cohomology with constant $S^1$-coefficients, we show next
that the cocycle remains non-trivial when we allow coefficients in the $G$-module of smooth $S^1$-valued functions
on $\frak k.$

\begin{lemma}  Any Lie groupoid cocycle in $H^2(G\ltimes \frak k,S^1)$ corresponding to the Lie algebroid
cocycle $\omega\in H^2(\frak g \ltimes \frak k, \mathbb R)$ is non-trivial.
\end{lemma} 
\begin{proof}  We show that   $\omega$ is non-trivial in the category of integrable Lie algebroid
cochains. It suffices to prove that it fails to be a coboundary at a single point $A\in \frak k$. Restricting first to the finite rank matrices  of  size $N\times N$ in $\mathfrak g$, the cocycle $\omega$ can be written as the coboundary of any 1-cochain $c_1$  of the form
$$c_1(A; X) = - \tr\, DX + b(A;X)$$
where $b$ is an arbitrary 1-cocycle, and this $c_1$  is the most general form such that $\omega =\delta c_1.$ 
The point $A= -D$, in the restriction to $N\times N$ matrices, is a fixed point for the $G$-action and therefore $X\mapsto b(-D;X)$ is a Lie algebra homomorphism to
the abelian algebra of complex numbers. But the only homomorphism for the Lie algebra of $N\times N$ matrices to $\CC$ is of the form $X\mapsto \lambda \text{tr}\, X$
for some complex number $\lambda$. This fixes the form of the 1-cocycle $b$ at the point $A=-D$, but this is not immediately helpful since $D$ does not belong to the space $\frak k$ in the limit $N\to \infty$.

Consider next the restriction of the cocycle $b$ to the abelian Lie algebra $\mathfrak h$ of $N\times N$ diagonal matrices. At this point we recall that we are really discussing the groupoid cohomology, so
all Lie algebroid cochains should be integrated to groupoid cochains.
The group cocycle $B$ corresponding to $b$ is a character of the group $H$ of invertible diagonal matrices. In fact, all the points $A= -tD$ for $0\leq t\leq 1$ are fixed
points of $H$ and therefore $h\mapsto B(-tD; h)$ is a character for all values of $t.$ But the characters form a discrete lattice in the space of maps $H\to \mathbb C^{\times}$
and by the continuity of $B$ we get $$B(-tD; h) = B(0; h) = B(-D; h) = \text{det}(h)^p$$ for some integer $p.$ Again, on the Lie algebra level, this means that
$b(0; X) = p\text{tr} \, X$ for $X\in \mathfrak{h}.$ Thus the most general form for the 1-cochain $c_1$ at $A=0$ and for $X\in \mathfrak h$ is $c_1(0;X) = \tr\, DX + p \tr \, X.$ This linear
function of $X$ does
not have a finite limit as $N\to\infty$ for any $p=p(N)$, and hence $\omega$ cannot be trivialized by any integrable 1-cochain.
\end{proof} 

\section{Integration of Lie algebroid cocycles}
In this section we shall extend some of the results on the construction of Lie group cocycles
by Wagemann and Wockel \cite{W} to the case of action Lie groupoids. This allows us to  integrate our cocycle $\omega$ explicitly
to a non-trivial Lie groupoid cocycle $c\in H^2(G\ltimes \frak k, S^1)$, which is locally smooth in an open identity neighbourhood of $G$. The latter determines a non-trivial $S^1$-gerbe on $G\ltimes \frak k$,
or equivalently, an abelian extension $\widehat G$ of $G$ by the group $C^\infty(\frak k, S^1)$.

Consider an action Lie groupoid $G \ltimes \frak k $, with a smooth right action of a Lie group $G$ on a smooth manifold $\frak k$. Let $\omega_p$ be a Lie algebroid $p$-cocycle in $H^p(\frak g\ltimes \frak k,\RR)$.
Let $\Sigma \subset G$ be a submanifold (possibly with boundary)  of dimension $p$ containing the identity in $G$ and let $A\in \frak k$ be a base point. Then we can define the pairing
$$ \langle\Sigma, \omega_p\rangle_A. $$
The pairing is defined by an integration over $\Sigma,$ viewing $\omega_p$ as a closed left-equivariant differential form $\tilde\omega_p$ on $G$ through
$$\tilde\omega_p(g; X_1, \dots ,X_p) = \omega_p(A\cdot g; X_1, \dots, X_p).$$
One should keep in mind that this depends on the choice of the base point $A.$  The cocycle $\omega_p$ is \it integral \rm if the form $\tilde \omega_p$ defines an 
integral cohomology class on $G$. Moreover, if $\frak k$ is connected the integrality does not depend on the choice of base point $A.$

Suppose that the homology groups  $H_k(G,\ZZ)$ vanish up to the dimension $k=p-1.$ Then any smooth integral $p$-cocycle $\omega_p$ defines a locally smooth groupoid cocycle $c_p.$  Given a pair of points $a,b\in G$ choose a path $s_1(a,b)$ from $a$ to $b.$
Given a triple of points $a,b,c \in G$ choose in addition a smooth singular 2-simplex $s_2(a,b,c)$ in $G$ with a boundary $b_1(a,b,c)$  consisting of the paths $s_1(a,b), s_1(b,c)$
and $s_1(c,a).$ This process can be continued to fix a singular $p$-simplex $s_p(a_1, a_2, \dots, a_{p+1})$ in $G$ for a given ordered subset $(a_1, \dots, a_{p+1}) \subset G.$

Let $g_1, \dots, g_p \in G$ and set $a_1 = e$ and $a_i = g_1g_2 \cdots g_{i-1}$ for $i=2, \dots, p+1.$ Then we can define a locally smooth $p$-cocycle on the Lie groupoid $G\ltimes \frak k$ by
$$c_p(A; g_1, \dots, g_p) =  e^{2\pi i\langle s_p(a_1, \dots, a_{p+1}), \omega_p\rangle_A}.$$

\begin{proposition} The cochain $c_p$  satisfies the cocycle condition
$$   c_p(A; g_1, \dots, g_p) \cdot c_p(A; g_1g_2, g_3, \dots, g_{p+1})^{-1}   \cdots  c_p(A; g_1, \dots, g_{p-1}, g_pg_{p+1})^{(-1)^p} $$
$$ \cdot  c_p(A\cdot g_1; g_2, \dots, 
g_{p+1})^{(-1)^{p+1}}=1,$$
\end{proposition}
 
\begin{proof} For the proof one simply needs to check that the sum of the relevant singular simplices $s_p$ has a vanishing boundary; this is seen by observing that the
boundary components of the different simplices match pairwise. The pairing of a singular cycle with the integral cocycle $\omega_p$ gives an integer which multiplied by
$2\pi$ and exponentiated gives $1.$ 
\end{proof} 

\begin{remark} 
One can relax the requirement on the homology groups of $G$. If $H_k(G,\ZZ)=0$ for $k\leq p-2$ then the above construction works if $c_p$ defines
 a  Cheeger-Simons differential character  of degree $p$ on $G.$ Recall that a differential character of degree $p$ with values in $S^1$ is a homomorphism $f$ from the space of $p-1$
 dimensional homology cycles to $S^1$ such that if $s_{p-1}$ is a boundary of a $p$-dimensional chain $s_p$ then $f(s_{p-1}) = \exp{2\pi i\int_{s_p} \omega}$, where $\omega$
 is a uniquely fixed closed integral differential form on $G.$ 
 \end{remark}
 \begin{remark}  If there is an obstruction coming from the lower homology groups $H_k(G,\ZZ)$ the above construction can in some cases be modified to define a groupoid
 cocycle $c_p.$ For example, assume that $H_1(G,\ZZ)$ is a free $\mathbb Z$-module and we want to define a cocycle of degree $3.$ Fix a set of 1-cycles 
 $w_1,\dots , w_n$ in $G$ such that the homology classes $[w_i]$ form a basis over $\mathbb Z$ in $H_1(G,\ZZ).$ Then a linear combination of the cycles $w_i$ defines 
 the unit element in $H_1(G,\ZZ)$ if and only if that linear combination vanishes strictly as a 1-cycle. Next starting from a pair $g_1,g_2$ of group elements, 
 form the 1-cycle $b_1(e,g_1,g_1g_2)$ as above. We can write $[b_1] = \sum_{i} n_i [w_i]$ for a set of integers $n_i.$ Replace now $b_1 $ by
 $b'_1= b_1 - \sum_i n_i w_i.$  Then $b'_1$ is a boundary of a 2-chain $s_2.$ For a triple $g_1,g_2,g_3$ we do this construction for each face of the simplices corresponding
 to pairs of group elements. In this way one obtains a closed cycle which can be filled to a tetraed provided that $H_2(G,\ZZ)=0.$ 
\end{remark}
 
 \begin{example}  The above remark is relevant in gauge theory. Suppose that $G= Map_0(S^2, SU(n))$ with $n\geq 3.$ Here $Map_0$ denotes the group of based smooth maps $f,$ namely $f(x)=e$ for a given
 point $x\in S^2.$ Then $H_1(G,\ZZ) = \pi_1(G) = \mathbb Z$ whereas $H_2(G,\ZZ) =0$ and
 $H_3(G,\ZZ) = \mathbb Z = H^3(G, \mathbb Z).$  In this case the differential character of degree $3$ is given by a closed integral 3-form on $G$ coming from an integral
 5-form on $SU(n)$ by transgression.  
  \end{example} 
  \begin{example}  In a similar way we can construct  a 2-cocycle for the transformation groupoid consisting of the group $G=Map_0(S^3, SU(n))$ acting on $\mathfrak{su}(n)$-valued connection
 1-forms on a trivial $G$-bundle on $M$ through  gauge transformations $A\mapsto A^g = g^{-1} Ag + g^{-1} dg.$  In this case $G$ is disconnected, the connected components are labelled by elements
 in $\pi_3(SU(n)) = \mathbb Z.$  Fix a function $a\colon S^3 \to SU(n)$ that generates $\pi_3(SU(n)).$ Then $g_1 \in G$ is homotopic to some $a^k$ for $k\in \mathbb Z$ and
 $g_2$ is homotopic to some $a^l$. Set $g'_1 = a^{-k} g_1$ and $g'_2 = a^{-l} g_2$ and define a cycle $b_1(e, g'_1, g'_1g'_2)$ as before. We can then choose a 
 filling $s_2(e, g'_1, g'_1g'_2)$ since the homology of $Map_0(S^3, SU(n))$ in dimension one vanishes for $n\geq 3.$  A Lie algebra 2-cocycle can be constructed as in \cite{FS, Mi85}.  This construction of the Lie groupoid cocycle is essentially the same as in \cite{Mi87}.

  \end{example} 

A normalized $p$-cocycle on an action Lie groupoid which is smooth in an open neighbourhood of the unit element and with values in $S^1$  defines a \v{C}ech cocycle of degree
$p-1.$  This fact was utilized by Neeb \cite{N} in showing that for the construction of a smooth central (or abelian) extension of a Lie group, it is sufficient to construct
a 2-cocycle that is globally defined but smooth only in a neighbourhood of the unit element. Below we provide a somewhat simplified proof of (part of) Neeb's result when the group is connected. Actually, in \cite{N} the connectedness assumption was replaced by an additional condition on the groupoid cocycle that is automatically satisfied when $G$ is connected. 

\begin{theorem}[Neeb \cite{N}] Let $c$ be a normalized group 2-cocycle on a connected Lie group $G$ with values in a smooth $G$-module $A$. Normalization means that $c(g,e) = c(e,g)$ equals the unit element $1\in A$ for all $g\in G.$ If $c$ is smooth in an open neighbourhood
 $U$ of the identity element $e$, then it defines a smooth abelian extension $$1\to A \to  \widehat{G} \to G \to 1.$$ Conversely, any smooth abelian extension of $G$ by $A$ determines a locally smooth 2-cocycle $c$.
\end{theorem} 

\begin{proof}  
Suppose first that there is a smooth abelian extension $\widehat{G}$ of $G$ by $A.$ Then we can choose a 
global (discontinuous)  section $\psi\colon G \to \widehat{G}$ such that $\psi$ is smooth in an open set $U$ containing $e.$ We define the 2-cocycle $c$ with respect 
to the section $\psi$ by
$$c(g_1, g_2) = \psi(g_1g_2) \psi(g_2)^{-1} \psi(g_1)^{-1}.$$
By construction it  satisfies the cocycle condition
$$ c(g_1, g_2) c(g_1g_2, g_3) = c(g_1, g_2g_3)[g_1 \cdot c(g_2, g_3)] $$
and is normalized. It is also smooth in an open neighbourhood $V$ of unity such that $V^2 \subset U.$

Choose an open cover of $G$ using
the left translated sets  $U_i = a_i U$ with $a_i \in G.$
Smooth local sections can be defined as
$$\phi_i(x) = \psi(a_i) \psi(a_i^{-1} x) = c(a_i, a_i^{-1} x)^{-1}\psi(x) $$
for $x\in U_i$ with transition functions $f_{ij}(x) = c(a_i, a_i^{-1} x)^{-1} c(a_j, a_j^{-1} x)$ on the overlaps $U_i\cap U_j.$  

The argument can then be reversed. Suppose that $c$ is a global cocycle, smooth in an open identity neighbourhood $V$ and normalized. 
Choose a smaller open neighbourhood $U$ of $e$ such that $U= U^{-1}$ and $U^2 \subset V.$ This can be achieved using the exponential map
from the Lie algebra to the Lie group. Choose an open cover of $G$ as before and define  $f_{ij}(x) = c(a_i, a_i^{-1}x)^{-1} c(a_j, a_j^{-1} x)$ for $x\in U_i \cap U_j.$ These functions
clearly satisfy the \v{C}ech cocycle property 
$$f_{ij}(x) f_{jk}(x) = f_{ik}(x)$$
on triple overlaps. However, since the factors defining $f_{ij}$ are not separately smooth,
we have to prove the smoothness in their domain of definition. Using the cocycle and the normalization property repeatedly we can write
\begin{align*}  f_{ij}(x) &= c(a_i, a_i^{-1}x)^{-1} c(a_j, a_j^{-1}x) \\
&=  c(a_i, a_i^{-1})^{-1} c(e,x)^{-1} a_i\cdot [c(a_i^{-1},x)]c(a_j, a_j^{-1} x)\\
&= c(a_i, a_i^{-1})^{-1} a_i\cdot c(a_i^{-1},x) c(a_j, a_j^{-1} x)\\
&= c(a_i, a_i^{-1})^{-1}  a_i\cdot [ c(a_i^{-1},x) a_i^{-1}\cdot c(a_j, a_j^{-1}x)]\\
&= c(a_i, a_i^{-1})^{-1}  a_i \cdot [c(a_i^{-1}, a_j) c(a_i^{-1} a_j, a_j^{-1} x)] . \end{align*}
The last expression shows that the transition function $f_{ij}$ is locally smooth in the argument $x.$  This is because $x\in U_i\cap U_j$ can be written as
$x= a_i y = a_j z$ for $y,z\in U.$ Now $a_i^{-1} a_j = yz^{-1} \in V$ and thus the middle factor is smooth in $x.$ Thus we have recovered the eventually topologically twisted
extension $\widehat G$ in terms of the local transition functions.

To  complete the proof we still need to check that the local 2-cocycles 
$$c_{ijk}(x,y) = c(a_k, a_k^{-1} xy)c(x,y) c(a_i, a_i^{-1} x)^{-1} [x\cdot c(a_j, a_j^{-1}y)^{-1}] $$
are smooth; here $x\in U_i, y\in U_j$ and $xy\in U_k.$ We suspect that this can be done without the connectedness assumption, but the details remain to be worked out. We refer to \cite{N} for a  proof of the global smoothness of the product, Theorem C.2 
and Proposition II.6. 

\end{proof}

\section{Families of cubic Dirac operators}

Let $\frak g$ be a finite dimensional quadratic complex Lie algebra with a basis $\{e_i\}$. The cubic Dirac operator $\mathcal D$ is an odd element of the quantum Weil superalgebra $\mathcal W(\frak g) = U(\frak g)\otimes Cl(\frak g)$ \cite{AM, KT, Ko, La}. Denoting by $e^i$ the dual basis with respect to the invariant bilinear form, $\gamma(e^i) = \gamma^i$ the generators of the Clifford algebra $Cl(\frak g)$ and $\lambda_{ijk}$ the structure constants, then
$$ \mathcal D  = \sum_{i=1}^{\text{dim}\ \frak g}e_i\gamma^i+\frac{1}{3}\widetilde{s}_i\gamma^i= \sum_{i=1}^{\text{dim}\ \frak g} e_i\gamma^i-\frac{1}{12}\sum_{i,j,k=1}^{\text{dim}\ \frak g} \lambda_{ijk}\gamma^i\gamma^j\gamma^k .$$

We would like to extend this construction to the complexified Lie algebra of our Banach Lie group $G$, 
$$\frak g_\CC= \Big\{(X_{ij})_{i,j\in\ZZ} \in \CC  \ \big | \ \sup_{i\in\ZZ}|X_{ii}| + \sum_{i,j\in\ZZ} |(i-j)X_{ij}|^2< \infty\Big\}.$$
For this it is necessary to modify the definition of $\mathcal W(\frak g)$  in order to make sense of the Clifford algebra. Recall that the space $\frak k$ of self-adjoint Hilbert--Schmidt operators is a 
real Hilbert space with the inner product
$\langle X,Y\rangle=\tr XY$. The complexification $\frak k_{\CC}$ is equipped with the
Hermitian inner product $\langle X,Y\rangle = \text{tr}\, X^* Y.$ The space  $\frak k_{\CC}$ is further an orthogonal $\frak g_{\CC}$-module under the adjoint action of $\frak g_{\CC}$ on $\frak k_{\CC}$. Thus it is possible to associate a quantum Weil algebra $\mathcal W(\frak g, \frak k) = U(\frak g_{\CC})\otimes Cl(\frak k_{\CC})$  to the pair $(\frak g,\frak k)$. Let $\{e_{ij}\}_{i,j\in\ZZ}$ denote the standard basis for both $\frak g_{\CC}$ and $\frak k_{\CC}$ with matrix elements  $(e_{ij})_{ \ell m} = \delta_{i \ell}\delta_{jm}$,  satisfying the commutation relations  of a level $k$ central extension,
$$[e_{ij},e_{\ell m}]=\delta_{j\ell}e_{im}-\delta_{im}e_{\ell j} + k\, \delta_{j\ell}\delta_{im} (\ell-i)$$ 
and with the dual basis $e^{ij}=e_{ji}$ with respect to the above inner product. There exists a \emph{formal} cubic Dirac operator in $\mathcal W(\frak g, \frak k)$ given by
$$\mathcal D = \sum_{i,j\in\ZZ} e_{ij}\gamma_{ji}+\frac{1}{3}\widetilde{s}_{ij}\gamma_{ji},$$
where $\widetilde{s}\colon \frak g_{\CC} \to \frak{spin(k_{\CC})} \subset Cl(\frak k_{\CC})$ is the spin lift of the adjoint representation $\text{ad}\colon \frak g_{\CC} \to \frak{so(k_{\CC})}$ and $\gamma_{ij} = \gamma(e_{ij})$ are the generators of $Cl(\frak k_{\CC})$ satisfying the Clifford relation $[\gamma_{ij} , \gamma_{k\ell} ]_+ = 2\delta_{i\ell}\delta_{jk}$. The expression for $\mathcal D$ is ill-defined as it involves infinite sums and it is therefore necessary to introduce certain operator subtractions to obtain a well-defined element in a suitable completion of $\mathcal W(\frak g, \frak k)$. In other words we seek to construct a $\widehat{\frak g}$-module $V \otimes \mathbb S$ and define the completion $\widehat{\mathcal W}(\frak g, \frak k)$ as the endomorphism algebra $\text{End}( V \otimes \mathbb S)$ in which the infinite sums in $\mathcal D$  make sense.

The space of Hilbert--Schmidt operators has a natural orthogonal decomposition $\frak k_{\CC} = \frak k_+\oplus \frak k_0 \oplus \frak k_-$, where $\frak k_\pm$ are the isotropic subspaces of strictly upper and lower triangular matrices and $\frak k_0 \cong \ell_2$ is the subspace of diagonal matrices. The spin representation of $Cl(\frak k_{\CC})$ is defined by 
$$\mathbb{S} = \text{S}_{\frak k_0} \otimes \bigwedge \frak k_-,$$
 where $\text{S}_{\frak k_0}$ is a fixed spin module of $Cl(\frak k_0)$. The latter is constructed analogously by first considering $\frak k_0$ as a real Hilbert space with the $\ell_2$-inner product. Then any complex structure $J\in O(\frak k_0)$ with $J^2=-1$ turns the complexification $\frak k_{0,\CC}=\frak k_0 \otimes_\RR \CC$ into a complex polarized Hilbert space $\frak k_{0,J}$, with the Hermitian inner product $\langle X,Y\rangle_\CC = \langle X,Y\rangle + i \langle X,JY\rangle$. The splitting $\frak k_{0,J} =\frak k_{0,J}^+\oplus \frak k_{0,J}^-$ is defined by the projection operators $P_\pm=\frac 12 (I\pm iJ)$ and the spin module is given by $\text{S}_{\frak k_0} = \bigwedge \frak k_{0,J}^-$. Any other complex structure that differs from $J$ by a Hilbert--Schmidt operator determines a unitarily equivalent representation. Furthermore, since two such complex structures are related by conjugation by an element in the restricted orthogonal group $O_{res}(\frak k_{0,J}^+\oplus \frak k_{0,J}^-)$, the possible inequivalent spin modules of $Cl(\frak k_0)$ are parametrized by $O(\frak k_0)/O_{res}(\frak k_{0,J}^+\oplus \frak k_{0,J}^-).$

In terms of the basis $\{e_{ij}\}_{i,j\in\ZZ}$ the spin representation $\widetilde{s}\colon \frak g_{\CC} \to \text{End}(\mathbb S)$ is formally given  by
$$\widetilde{s}_{ij} =\frac 12 \sum_{\ell\in \ZZ} \gamma_{i\ell}\gamma_{\ell j} .$$
This sum is of course divergent but it can be regularized by applying the standard normal ordering prescription. In the following we shall restrict to operators in $\frak g_{\CC}$ with zero entries outside of a block of size $2N$ so that all sums become finite. We  pass to the limit $N\to \infty$ in the strong operator topology, only in the final regularized expressions. Let
 \begin{align*} s_{ij} &= \frac 12  \sum_{\ell}  :\gamma_{i\ell}\gamma_{\ell j}: = \begin{cases}\frac 12  \sum_{\ell}  \gamma_{i\ell}\gamma_{\ell j} \,\,\, \ \ \mbox{if}\  i\neq j \\ \frac 12\sum_{\ell<i} \gamma_{i\ell}\gamma_{\ell i}  - \frac12 \sum_{\ell > i} \gamma_{\ell i}\gamma_{i \ell} \,\,\, \ \ \mbox{if}\  i= j\end{cases} \\
&= \widetilde{s}_{ij} -(N-i+\frac12)\delta_{ij} .
\end{align*}

As a result of normal ordering, the spin operators satisfy the commutation relations 
\begin{align*}
[ s_{ij} ,  s_{\ell m} ] &= \delta_{j\ell}s_{im}-\delta_{im}s_{\ell j} +\delta_{j\ell}\delta_{im}(N-i+\frac12)-\delta_{im}\delta_{\ell j}(N-\ell+\frac12) \\
&= \delta_{j\ell}s_{im}-\delta_{im}s_{\ell j} +\delta_{j\ell}\delta_{im}(\ell-i),
\end{align*}
where the central term is precisely $\omega(s_{ij},s_{\ell m})$. Since the dependence on $N$ cancels out it is possible to extend these operators  to the whole Lie algebra by taking the limit $N\to \infty$ and this yields a projective representation $\mathbb S$ of $\frak g$ of level $1$. One easily checks that the following commutation relations hold,
\[  [ s_{ij} ,  \gamma_{\ell m} ] =  \delta_{j\ell}\gamma_{im}-\delta_{im}\gamma_{\ell j}, \ \ \   \ \ \ 
\sum_{\ell , m}  [ s_{ij}, s_{\ell m} \gamma_{m\ell} ] =  (j-i)\gamma_{ij}.
\]

With our choice of the normal ordering, $s_{ij}v=0$ for all $i\leq j$ when $v$ is the vector in the
`vacuum sector', that is, it is of the form $v=w\otimes 1 \in \text{S}_{\frak k_0} \otimes \bigwedge \frak k_-.$ On these vectors $\gamma_{ij}v =0$ for $i<j.$ In particular, the weight of the spin module is equal to zero. It should be kept in
mind that the highest weight space is not spanned by a single vector but it is the infinite-dimensional
irreducible module $\text{S}_{\frak k_0}$ of the Clifford algebra.

Next we would like to construct a Verma module  for the central extension $\widehat{\frak g}_\CC$. There is a natural triangular decomposition $\widehat{\frak g}_\CC =\frak g_+\oplus \frak h\oplus \CC K\oplus \frak g_-$, where $\frak g_\pm$ correspond to strictly upper respectively lower triangular matrices and $\frak h \cong \ell_\infty$ is the infinite dimensional Cartan subalgebra of diagonal matrices. Let $\Lambda = \ell_\infty^*$ denote the weight space. Starting with the universal enveloping algebra $U(\widehat{\frak g}_\CC)$ and a weight $\lambda \in \Lambda$, the Verma module is defined by the quotient $ V_{(\lambda, k)}=U(\widehat{\frak g}_\CC)/I(\lambda,k)$ where $I(\lambda,k)$ is the left ideal generated by $\frak g_+$ and elements $h-\lambda(h)\textbf{1},h\in \frak h$ and $K- k\textbf{1}$. Let $v_\lambda$ denote the image of the identity element $\bf{1}$ of $U(\widehat{\frak g}_\CC)$ in the quotient, then $e_{ii}v_\lambda = \lambda_i v_\lambda$ with $\lambda_i = \lambda(e_{ii})$ and $e_{ij}v_\lambda = 0 $ for $i<j$.

If the representation of the Lie algebra $\widehat{\frak g}_\CC$ can be integrated to a representation of
$\widehat G$, then the components $\lambda_i$ have to belong to $\mathbb{Z}.$
In a unitary representation we have in addition the positivity constraints 
$$0\leq||e_{ji}v_\lambda||^2=(v_\lambda, e_{ij}e_{ji}v_\lambda) = (v_\lambda,(e_{ii}-e_{jj}+k(j-i))v_\lambda), \ \ \ \ i \leq j,$$ 
which implies 
$$\lambda_i-\lambda_j-k(i-j)\geq0, \ \ \ \ i\leq j.$$
The problem of integrating unitary representations of Banach Lie algebras is far from trivial and while there exists 
integrability criteria extending Nelson's famous criterion \cite{Me, VTL}, we shall consider a different route to constructing integrable $\widehat{\frak g}_\CC$-modules.
 
Recall that there is a natural homomorphism $X\mapsto  \text{ad}_X$ from $\frak g$ to  the Lie algebra $\frak u_{res}(\frak g_+ \oplus \frak g_-)$ by the adjoint action of $\frak g.$  Here $\frak g_+$ denotes the span of the generators $e_{ij}$ with $i\leq j$ and $\frak g_-$ by the generators $e_{ij}$ with $i>j.$ In this
action $\text{ad}_{e_{ii}}$ is represented by the matrix 
$$\text{ad}_{e_{ii}} = \sum_{m\in \ZZ} (E_{im} -E_{mi})$$
where $E_{ij}$ is the diagonal linear operator $E_{ij} \cdot e_{mn} = \delta_{im}\delta_{jn}  e_{mn}.$
The Lundberg cocycle $\omega_L(X,Y) = -\frac12 \text{tr}\, X[\epsilon, Y]$ on ${\frak u}_{res}(\frak
g_+ \oplus \frak g_-)$ induces then the level two cocycle $\omega(X,Y) = \omega_L(\text{ad}_X, \text{ad}_Y) = 
2\text{tr}\, X [D,Y]$ on $\frak g.$ Thus any integrable representation of $\hat{\frak u}_{res}$ gives 
an integrable representation of $\hat{\frak g}.$ In particular, the integrable highest weight representations of $\hat{\frak u}_{res}$ give (in general reducible) highest weight representations
of $\hat{\frak g}.$ The highest weight of the former is given by the set $\mu(E_{ij}) = \mu_{ij}$ 
of integers. If the level of the ${\frak u}_{res}$ representation is $k$ then $|\mu_{ij}| \leq k.$
Only a finite number of the $\mu_{ij}$'s can be nonzero. The weight $\mu$ of $\frak u_{res}$
determines a weight $\lambda$ of $\frak g,$ 
$$\lambda_i = \lambda(e_{ii}) =\mu(\text{ad}_{e_{ii}})= \sum_{m\in \ZZ} (\mu_{im} - \mu_{mi})$$
which is also integral and has only a finite number of nonzero entries.

The doubling of the level from the restriction $\frak g \subset \frak u_{res}(\frak g_+ \oplus
\frak g_-)$ results from the fact that the adjoint representation of $\frak g$ is real, so the
relevant embedding is actually $\frak g \subset \frak o_{res} \subset \frak u_{res}.$ 
The space $\frak g$ of skew-adjoint operators should be considered as a real Hilbert space
and its complexification $\frak g_{\mathbb C}= \frak g_+ \oplus \frak g_-$ defines the polarization 
in the complexification $\frak o_{res, \mathbb C}.$ The geometric construction of representations
of $\frak u_{res}$ acting on holomorphic sections of a determinant bundle on a homogeneous
space of $U_{res}$ can be refined for $\frak o_{res}$ as a representation acting on Pfaffians
which are square roots of determinants \cite{PS}. The square root
construction leads to a basic representation of $\frak o_{res},$  with the value of the center
equal to 1, which then corresponds to the basic cocycle $\text{tr}\, X[D,Y]$ of $\frak g.$ 

The ordering of the integers $\mu_{ij}$ depends on the choice of simple root vectors; we can fix
a linear ordering of the index pairs  $(ij) > (mj)$ for $i >m$ and $(ij)> (im)$ for $j < m.$
In this ordering we can fix the simple root vectors as $E_{(ij), (i,j-1)}$ and $E_{(ij),(i+1,j)}$ in 
${\frak u}_{res}(\frak g_+
\oplus \frak g_-),$ following the standard notation that $E_{ab}$ means the matrix with all entries
equal to zero except the entry at the position $(ab)$ equal to 1. In this ordering the elements
$$\text{ad}_{e_{ij}} = \sum_{m\in \ZZ} (E_{(im),(jm)} - E_{(mj),(mi)})$$
are linear combinations of positive root vectors for $i<j.$ 

Thus   any integrable unitary  highest weight representation of $\hat{\frak u}_{res}(\frak g_+ \oplus \frak g_-)$  indeed gives by restriction an integrable unitary highest weight representation $V_\lambda$ of $\hat{\frak g}$. The cubic Dirac operator can now be defined by
$$\mathcal D = \sum_{i,j} e_{ij}\gamma_{ji}+\frac{1}{3}s_{ij}\gamma_{ji},$$
acting on the $\widehat{\frak g}_\CC$-module $H = V_\lambda \otimes \mathbb S$. The latter is generated by the highest weight subspace, denoted $v_\lambda$, which carries an irreducible representation of the Clifford algebra $Cl(\mathfrak{k}_0)$. The expression for $\mathcal D$ is however still problematic in the limit $N\to \infty$ and requires further regularization. We introduce an additional operator subtraction, 
$$\mathbb D = \sum_{i,j} e_{ij}\gamma_{ji}+\frac{1}{3} : s_{ij}\gamma_{ji}:,  $$
where the normal ordering is defined as
$$ \sum_{i,j} : s_{ij}\gamma_{ji} : = \sum_{i> j} s_{ij}\gamma_{ji} + \sum_{i\leq j}  \gamma_{ji}s_{ij} .$$
Note that here the normal ordering is not simply a subtraction by an (infinite) constant because of the nonzero commutators $[s_{ij},\gamma_{ji}] = \gamma_{ii} - \gamma_{jj}.$ Let $t_{ij}=e_{ij} + s_{ij}$, then we have

\begin{theorem} The cubic Dirac operator $\mathbb D$ extends to a well-defined  unbounded self-adjoint operator in the completion $\widehat{\mathcal W}(\frak g, \frak k) = \text{End}(V_\lambda \otimes \mathbb S)$ with a dense domain $$\text{Dom}(\mathbb D)=(V_\lambda \otimes \mathbb S)^{pol} \subset V_\lambda \otimes \mathbb S,$$ 
consisting of polynomials in the generators $t_{ij}$ and $\gamma_{ij}$ applied to the  highest weight subspace $v_{\lambda}$, and with the kernel 
$$Ker(\mathbb D) = \text{S}_{\frak k_0}. $$
\end{theorem}
\begin{proof} 
We begin by computing some relevant commutators, where throughout we consider a finite approximation with indices  in the range $[-N, N]$. Rewriting the normal ordered term $$\sum_{i,j} : s_{ij}\gamma_{ji} :\ = \sum_{i,j}  s_{ij}\gamma_{ji} + \sum_{i\leq j}(\gamma_{jj}-\gamma_{ii})= 
\sum_{i,j}  s_{ij}\gamma_{ji}  + 2\sum_{i}  i\gamma_{ii},$$ we have
 \begin{align*}
\sum_{\ell,m}  [:s_{\ell m}\gamma_{m\ell}:,s_{ij}] &= (i-j)\gamma_{ij} +2 (i-j) \gamma_{ij} \\
   & = 3(i-j) \gamma_{ij}.
  \end{align*}
Furthermore,
 \begin{align*}
\sum_{\ell, m}  [ s_{\ell m}\gamma_{m\ell }, \gamma_{ij}]_+ &= \sum_{\ell, m}s_{\ell m}   [\gamma_{ij}, \gamma_{m\ell }]_+ + \sum_{\ell, m} [ \gamma_{ij} , s_{\ell m}] \gamma_{m\ell}  \\
&=  2 s_{ij} + \sum_{\ell,m} (\delta_{\ell j} \gamma_{im} -\delta_{mi} \gamma_{\ell j} ) \gamma_{m \ell} \\
&= 2 s_{ij} + \sum_{m}  \gamma_{im}\gamma_{mj} -\sum_{\ell} \gamma_{\ell j}\gamma_{i \ell} \\
&= 6 s_{ij}  + 4 \delta_{ij} (N - i+\frac12)  - 2(2N+1) \delta_{ij}  \\
&= 6 s_{ij} - 4i\cdot \delta_{ij} 
 \end{align*}
 where we have used $2s_{ij} = \sum_m \gamma_{im} \gamma_{mj} - 2(N  - i +\frac12)\delta_{ij} .$  Next, using the expression $\sum_{i,j} : s_{ij}\gamma_{ji} :  =
\sum_{i,j}  s_{ij}\gamma_{ji}  + 2\sum_{i}  i\gamma_{ii}$, we obtain
 $$[ \sum_{\ell, m}   :s_{\ell m}\gamma_{m\ell }:, \gamma_{ij}]_+ =6 s_{ij}.$$ 
 In addition, $\sum_{\ell, m} [e_{\ell m}\gamma_{m \ell}, \gamma_{ij}]_+  = 2e_{ij}$ and thus
 $$ [\mathbb{D}, \gamma_{ij}]_+  = 2(e_{ij} + s_{ij}) = 2 t_{ij}.$$
 A similar computation, for a level $k$ central extension of the Lie algebra generated by $\{e_{ij}\}$, gives
 $$[\mathbb{D}, t_{ij}] = (i-j) (k+1) \gamma_{ij}.$$ 
 It follows that
 $$[\mathbb{D}^2, \gamma_{ij}] = -[\mathbb{D}, \gamma_{ij}]_+ \mathbb{D} + \mathbb{D} [\mathbb{D}, \gamma_{ij}]_+\\
 = -2t_{ij}  \mathbb{D} + \mathbb{D} \cdot 2t_{ij} = 2(i-j)(k+1) \gamma_{ij}$$
 and similarly
 $$[\mathbb{D}^2, t_{ij} ] = 2(i-j)(k+1) t_{ij} .$$
The final result does not depend on the cutoff parameter $N$ and hence we can pass to the limit $N\to \infty$ in the strong operator topology.  This proves that $\mathbb D$ is a well-defined operator in the algebra $End(V_\lambda \otimes \mathbb S)$.
  
 From the above commutation relations we can write a simple expression for the square of the operator
 $\mathbb{D},$
 $$\mathbb{D}^2 = \sum_{i,j\in \ZZ} :e_{ij}e_{ji}: +  \sum_{i,j\in \ZZ}  i : \gamma_{ij}\gamma_{ji}: $$
 where
 $$ :e_{ij}e_{ji}: = \begin{cases} e_{ij}e_{ji} \text{ if } j \leq i \\
 e_{ji}e_{ij} \text{ if } i < j \end{cases}. $$
 To show that this is the right expression one only needs to show (by a simple computation) that
 the commutator of $\mathbb{D}^2$ with $\gamma_{ij}$ and $t_{ij}$ agrees with the 
 already given formulas above and that the action of $\mathbb{D}^2$ on the highest weight subspace is correct.
 The last statement is true since $\mathbb{D}$ acting on the highest weight vectors $v_\lambda$ is
 a multiplication by $\sum_i \gamma_{ii}\lambda_i$ and the square of this is $\sum_i \lambda_i^2$,
 which agrees with the action of the quadratic expression above. 
Note that we are only considering representations of the Lie algebra which can be exponentiated to projective
representations of the group $G,$ so the numbers $\lambda_i$ must all be integers. Since the sum 
$\sum_i \lambda_i^2$ converges, it follows then that only a finite number of the components $\lambda_i$ are nonzero. 

The difference $C= \mathbb{D}^2 - \sum_i 2(k+1) i t_{ii}$ plays the role of the Casimir operator of $\widehat{\frak g}_\CC$
since it commutes with all the generators $\gamma_{ij}, t_{ij}.$ Its value on the vacuum subspace
$v_\lambda$ is equal to $\lambda^2 - \sum_i 2(k+1) i \lambda_i$  and this converges by the 
remark above for integrable representations. A further immediate consequence of the above commutation relations and the identity $\langle\mathbb D^2 v_\lambda,v_\lambda \rangle = \langle v_\lambda , \mathbb D^2 v_\lambda\rangle$ is that the operator $\mathbb D^2$ is unbounded, non-negative and symmetric on the dense domain $(V_\lambda \otimes \mathbb S)^{pol}$. It is also diagonalizable in a basis given by ordered monomials in the generators $\gamma_{ij}, t_{ij}$ acting on $v_\lambda$, which implies that $\mathbb D^2$ can be extended to a self-adjoint operator.

Finally, note that the eigenvalues of $\mathbb{D}^2$ have an infinite degeneracy due to the fact that the diagonal elements $\gamma_{ii}$ commute with $\mathbb{D}^2$. In particular, the vacuum $v_{\lambda}$ is
 not a single vector but an irreducible representation of the Clifford algebra $Cl(\frak k_0)$ generated by the
 elements $\gamma_{ii}.$ All these vectors have the same weight for the Cartan sublagebra
 spanned by the elements $t_{ii}$. This observation combined with the above commutation relations and the action of $\mathbb{D}^2$ on the highest weight subspace $v_\lambda$ imply that the kernel of $\mathbb{D}^2$ is the infinite-dimensional subspace $\text{S}_{\frak k_0}$.
The theorem follows by taking the square root of $\mathbb D^2$.

 \end{proof}

 Next we introduce a family of operators $\mathbb{D}_A$ parametrized by self-adjoint Hilbert--Schmidt operators
 $A\in \frak k,$ which play the role of gauge connections in our setting. We let $$\mathbb{D}_A = \mathbb{D} + (k+1)\sum_{i,j\in\ZZ} \gamma_{ij}A_{ji}.$$
 The connection $A$ transforms under a conjugation by $g\in G,$ represented as $\hat g$ in the projective representation of $G$ of level $k+1,$  as $A^g = g^{-1}A g + g^{-1}[D, g].$
Next we compute the square
\begin{align*}
\mathbb{D}_A^2 &= \mathbb{D}^2 + (k+1)[\mathbb{D}, \sum_{i,j\in\ZZ}\gamma_{ij} A_{ji}]_+ + (k+1)^2\sum_{i,j\in\ZZ} A_{ij}^2\\
&= \mathbb{D}^2 + 2(k+1)\sum_{i,j\in\ZZ} A_{ij}t_{ji} + (k+1)^2 A^2. 
\end{align*}
 
 We want to gather information about the spectrum of $\mathbb{D}_A.$ First, we observe that
 $A$ can be diagonalized by a 'gauge transformation' $g\in G.$ This follows from the fact that
 the unbounded self-adjoint operator $D_A = D +A$ can be diagonalized by a unitary transformation $g\in U(H).$
 But after diagonalization, $$g^{-1} D_A g = D + A' = D + g^{-1}A g + g^{-1}[D,g]$$ and since
 $g^{-1} A g$ is Hilbert--Schmidt by the $L_2$-property of $A$ and the off diagonal
 part of $A'$ vanishes, we conclude that $[D,g]$ is also Hilbert--Schmidt and so $g\in G.$ 
 
 We can thus assume that $A$ is diagonal without loss of generality, with diagonal values
 $A_{ii} = \mu_i.$ In this case
 \begin{align*}
 [\mathbb{D}^2_A , \gamma_{ij}] &= 2(k+1)(i-j+\mu_i -\mu_j) \gamma_{ij} ,\\
 [\mathbb{D}^2_A, t_{ij}] &= 2(k+1)(i-j+\mu_i-\mu_j)t_{ij}.\end{align*}
 But now
 $$(\mathbb{D}^2_A -\mathbb{D}^2)v_\lambda = \left((k+1)^2\sum_i \mu_i^2 + 2(k+1)\sum_i \mu_i \lambda_i\right)v_\lambda,$$
 that is, the vacuum eigenvalue of $\mathbb{D}^2_A$ is equal to $(\lambda + (k+1)\mu)^2.$ 
 The eigenvalues of $\mathbb{D}_A$ are then square roots of the eigenvalues of $\mathbb{D}^2_A.$
 
 The Hilbert space $H = \mathbb{V}_{\lambda} \otimes \mathbb{S}$ is spanned by the
 vectors
 $$v_{\lambda}^{(i), (j)}= 
t_{(i),(j)}\gamma_{(i),(j)} v_{\lambda}= 
 t_{i_1j_1}\cdots t_{i_p j_p} \gamma_{i_{p+1}, j_{p+1}} \cdots \gamma_{i_qj_q} v_{\lambda}$$
 with $i_s > j_s$ for $s=1,2,\dots q.$ 
 Using the commutation relations of $\mathbb{D}^2_A$ with $\gamma_{ij}, t_{ij}$ we conclude
 that
 $$ \mathbb{D}^2_A v_{\lambda}^{(i), (j)}= \left( (\lambda + (k+1)\mu)^2 \\
 + \sum_s  2(k+1) (i_s  - j_s  + \mu_{i_s} - \mu_{j_s}) \right) v_{\lambda}^{(i), (j)}.$$
 The sum over the index $s$ can be negative, so the eigenvalue of $\mathbb{D}^2_A$ can be zero for some parameters $\mu.$ However, for any given vector $\mu$ there can be at most
 a finite number of the integer sequences $(i_s > j_s)$ for which this is the case. The reason for this is
 that because $\mu^2$ is convergent, it follows that $\mu_i \to 0$ as $i\to \pm \infty$ and therefore for large values
 of the indices the positive integers $i_s -j_s$ dominate over the small, potentially negative, numbers
 $\mu_{i_s} - \mu_{j_s}.$ 
 
 By the equivariance property
 $$ \hat g^{-1}\mathbb{D}_A \hat g = \mathbb{D}_{A^g}$$
 the action of the isotropy group $G_A\subset G$ at the point $A$ commutes with the quantum operator
 $\mathbb{D}_A$ and therefore leaves  its spectral subspaces invariant. For a generic $A,$ for which the differences of the numbers $\mu_i$ are not in $\mathbb{Z},$ the isotropy group is just the group
 of diagonal matrices in $G.$ In the general case, the Lie algebra of the isotropy group contains
 also those elements $t_{ij}$ for which 
 $$i-j + \mu_i - \mu_j =0 .$$ 
 As we remarked earlier, there are only a finite number of these pairs of indices with $i\neq j,$ and thus the isotropy group at $A\sim \mu$
 is a finite rank perturbation of the group of diagonal matrices. 
 
The kernel of $\mathbb{D}_A$, which is the same as the kernel of $\mathbb{D}^2_A$,
carries a representation of the tensor product of the Clifford algebra
$Cl(\mathfrak{k}_0)$ and the enveloping algebra of the Cartan subalgebra $U(\mathfrak{h}).$ The vacuum sector carries an irreducible
representation of this algebra. Since the kernel of $\mathbb{D}_A$ consists of a finite number of
operators $t_{(i),(j)}\gamma_{(i),(j)}$ applied to the vacuum representation, the kernel is a finite sum
of irreducible representations of $ U(\mathfrak{h}) \otimes Cl(\mathfrak{k}_0) .$ 
We have thus proven
\begin{theorem} The kernel of the operator $\mathbb{D}_A$ is a finite direct sum of irreducible
representations of the algebra $\mathcal{W}(\frak g_A)= U(\mathfrak{g}_A)\otimes Cl(\mathfrak{g'}_A) $, where $\mathfrak{g'}_A$ is the
intersection of the
Lie algebra $\mathfrak g_A$ of the isotropy group $G_A\subset G$ at the point $A$ with Hilbert--Schmidt operators.
\end{theorem} 
\begin{remark}After a gauge transformation $A^g = g^{-1}Ag +g^{-1}[D,g]$ the operator $A$ can be brought
to a diagonal form and the Lie algebra of $\mathfrak{g}_A$ is then the commutative Lie algebra
$\mathfrak{h}$ extended by a finite dimensional algebra of finite matrices.
\end{remark}

\end{document}